\documentclass{article}
\usepackage{amssymb}
\usepackage{enumerate,verbatim}
\usepackage{amsmath}
\usepackage{amsfonts}
\usepackage{amsthm}
\usepackage{xcolor}
\usepackage{amsmath}
\usepackage{amsfonts}
\usepackage{array}
\usepackage{enumerate}
\usepackage{graphicx}
\usepackage{amssymb}
\usepackage{amsthm}
\usepackage{xcolor}

\newtheorem{remark}{Remark}
\newtheorem{definition}{Definition}
\newtheorem{example}{Example}
\newtheorem{corollary}{Corollary}
\newtheorem{lemma}{Lemma}
\newtheorem{proposition}{Proposition}

\begin{document}

\date{}

\author{Giuseppe Gaeta \\ {Dipartimento di Matematica, Universit\`a degli Studi di Milano,} \\ {v. Saldini 50, I-20133 Milano (Italy)}; \\ {INFN, Sezione di Milano} \\ and \\
{SMRI, I-00058 Santa Marinella (Italy)} \\ {\tt giuseppe.gaeta@unimi.it} \\ {} \\
Sebastian Walcher \\ {Fachgruppe Mathematik, RWTH Aachen,} \\ {52056 Aachen (Germany)} \\ {\tt walcher@mathga.rwth-aachen.de} }

\title{On gauge transforms of autonomous ordinary differential equations}
\maketitle

\begin{abstract}
\noindent
The notion of gauge transform has its origin in Physics (Field Theory). In the present note we discuss -- from a purely mathematical perspective -- special gauge transforms of autonomous first order ODE's and their special properties. Particular attention is given to the  problem of identifying those nonautonomous ODE's which are gauge transforms of autonomous systems. 
\\ { }  
\\ { } 
\end{abstract}

\newpage

\section{Introduction and basics}

Let an autonomous ODE
\begin{equation}\label{odeauto}
\dot x=f(x)
\end{equation}
be given on some open set in $\mathbb R^n$, with analytic\footnote{Basic notions and some results extend to sufficiently differentiable systems. Analyticity is required for the identification problem (Section 4).} right hand side. As usual, the dot symbolizes the derivative with respect to $t$.
Moreover we consider a curve in $GL(n,\mathbb R)$, i.e. an analytic (or sufficiently differentiable) map $\mathbb R\to GL(n,\mathbb R)$, $t\mapsto A(t)$, with $A(t_0)=A_0$. For the sake of simplicity (and with no loss of generality) we will take $t_0=0$ in the following.
\begin{definition}\label{defgt}
\begin{enumerate}[(a)]
\item The nonautonomous equation 
\begin{equation}\label{weirdgt}
\dot y=f^*(t,y):=\dot A A^{-1}y+Af(A^{-1}y)
\end{equation}
will be called the {\em gauge transform} of \eqref{odeauto} by $A(t)$.
\item Generally a nonautonomous equation $\dot x=q(t,x)$ will be called a gauge transform whenever there exist an autonomous equation $\dot x=f(x)$ and a curve $A(t)$ in $GL(n,\mathbb R)$ such that $q=f^*$.
\end{enumerate}
\end{definition}
We note  some elementary properties.
\begin{remark}\label{correm}{\em 
As to the relation between solutions of \eqref{odeauto} and of \eqref{weirdgt}, first recall
\[
\frac{d}{dt}\left(A^{-1}(t)\right)=-A^{-1}(t)\dot A(t)A^{-1}(t).
\]
Now, if $z(t)$ solves \eqref{odeauto}, look at the transformed solution
\[
w(t)=A(t)z(t).
\]
A straightforward calculation yields
\[
\begin{array}{rcl}
\dot w&=&\dot A z+A\dot z=\dot A z + Af(z)\\
   &=& \dot A A^{-1}w + A\left(f(A^{-1}w)\right).
\end{array}
\]
So, $w(t)$ solves \eqref{weirdgt}.
}
\end{remark}
\begin{remark}\label{secordrem}{\em 
In Mechanics, one is  interested in second order equations: Given
\[
\ddot x=h(x,\dot x);\qquad\text{equivalently  }\begin{pmatrix}\dot x\\ \dot y\end{pmatrix}=\begin{pmatrix} y\\ h(x,y)\end{pmatrix},
\]
consider a solution $z(t)$, thus $\ddot z(t)=h(z(t),\dot z(t))$. With $A(t)$ as above, and $w(t)=A(t)z(t)$, we find $\dot w(t)=\dot A(t(z(t)+A(t)\dot z(t)$, thus we arrive at
\[
\begin{pmatrix}w(t)\\ \dot w(t)\end{pmatrix}=\begin{pmatrix}A(t)&0\\ \dot A(t)&A(t)\end{pmatrix}\begin{pmatrix}z(t)\\ \dot z(t)\end{pmatrix},
\]
which brings us back to the situation of Definition \ref{defgt}.
}
\end{remark}

\section{Motivations}
\subsection
{ Motivation and examples from Physics}
\begin{itemize}
\item Gauge theories are a cornerstone of modern Physics. Gauge transformations were first considering in the context of the Electro-Magnetic field, both classical and quantum \cite{LLFT,Mes}, and their use -- or the requirement of gauge invariance and the study of spontaneous gauge symmetry breaking -- became an essential tool in general Field Theory (defined by a variational principle), in particular once the extension to non-Abelian group was understood \cite{YM}. See e.g. the textbook \cite{CM} for a physicist's perspective, or \cite{Ble} for a mathematical approach.
\item A direct obvious motivation comes from Classical Mechanics: If we consider a change of reference frame (e.g. passing from a Newtonian frame to one with the same origin but rotating w.r.t. this, maybe with a non-constant angular speed) we are indeed led to consider a gauge transformation \cite{CCL,EGH}; more generally, our formalism allows to discuss transitions between different reference frames; see also the Example below.

\item A more subtle physical motivation is the following: It is well known that when considering systems with periodically time-varying parameters one finds, in several contexts, a \emph{geometric phase} after one cycle \cite{Berry,Hannay,SW,Simon}. It was observed by Zak \cite{Zak} (see also \cite{Pati}) that if the parameters vary non-periodically, but periodically up to a gauge transformation, then one is still able to analyze the system. It is natural to extend the concept of periodicity (or, in the extreme case, constancy) up to gauge transformation in general dynamics.  

\item Finally, we recall that attention of physicists was already called to the appearance of gauge transformations in simple dynamical systems \cite{WZ}; for obvious physical reasons, however, this attention was focused on \emph{Hamiltonian} systems, while here we discuss general dynamical systems.

\end{itemize}

\begin{example}{\em Rotating reference frame 
\begin{itemize}
\item Consider two cartesian reference frames, say with the same origin for the sake of simplicity; we choose one of these, denoted as $F$, to be a (fixed) Newtonian one, and the other, denoted as $M$, to be a mobile one  rotating w.r.t. $F$ with angular velocity ${\bf A}$. Given a vector quantity  ${\bf v} (t)$, we have
\begin{equation} \label{eq:velMF} \left[ \frac{d {\bf v}}{d t} \right]_F \ = \  \left[ \frac{d {\bf v}}{d t} \right]_M \ + \ {\bf A} \wedge {\bf v} \ . \end{equation}
Correspondingly, we will denote total time derivatives in the two reference frames as $D_t^{(F)}$ and $D_t^{(M)}$ respectively, and the relation above can be written as 
\begin{equation}\label{eq:velMFb} D_t^{(F)} \ = \ D_t^{(M)} \ + \ {\bf A} \wedge \ . \end{equation}
Note that here rotations are completely arbitrary in three-dimensional space.
\item Consider now the case where dynamics takes place in the $(x_1,x_2)$ plane, and the (in general, non-uniform)  rotation is also within this plane, with angular velocity $ \omega  = \omega (t)$. Passing to the mobile frame we introduce new variables ${\bf w} = (w_1,w_2)$ with 
$$ {\bf w} \ = \ R[\vartheta (t)] \ {\bf x} \ , \ \ \ R[\vartheta] \ = \ \begin{pmatrix} \cos \vartheta & - \sin \vartheta \\ \sin \vartheta & \cos \vartheta \end{pmatrix} \  $$ 
{where $$ \vartheta = \vartheta (t) = \vartheta(0) + \int_0^t \omega (\tau) \ d \tau \ . $$} 
\item Let us consider a simple first order linear system defined in the fixed frame, 
$$ \dot{{\bf x}} \ = \ K \ {\bf x} $$
where ${\bf x} = (x_1,x_2)$ and $K$ is a constant real matrix. 
By Remark \ref{correm}, in the new variables we have 
$$ \frac{d}{d t} {\bf w} \ = \ \frac{d}{d t} ( R \, {\bf x} ) \ = \ \dot{R} \, {\bf x} \ + \ R \, \dot{\bf x} \ = \ \dot{R} \, R^{-1} \, {\bf w} \ + \ R \, K \, R^{-1} \, {\bf w} \ := \ L \, {\bf w}.  $$
Here $L$ is an explicitly time-dependent real matrix, $L =  \dot{R} R^{-1} + R K R^{-1}$, unless $\dot{R}=0$.
\item We turn to nonlinear systems. Consider
$$ \dot{{\bf x}} \ = \ K \ {\bf x} -|{\bf x}|^2 {\bf x}.$$
A straightforward calculation yields the gauge transformed system
$$ \frac{d}{d t} {\bf w} \ =  \ L \, {\bf w} -|{\bf w}|^2 {\bf w} $$
with $L$ as above. The fact that \emph{only} the linear part has changed depends on the very special form of the nonlinear one.
\item For a different nonlinear example consider 
$$ \dot{{\bf x}} \ = \ K \ {\bf x} + \Phi({\bf x}) $$
with, say, $\Phi = ( x_1 x_2 , x_2^2)$. In this case the transformed system is
$$ \frac{d}{d t} {\bf w} \ =  \ L \, {\bf w} + \Psi( {\bf w})  $$
with $L$ as above and 
$$ \Psi \ = \ \left( w_1 \, w_2 \ + \ (w_1^2 - w_2^2) \, \sin \vartheta \, \cos \vartheta \ , \ 
(w_1 \, \sin \vartheta \ - \ w_2 \, \cos \vartheta )^2 \right) . $$
\end{itemize}
}
\end{example}
\subsection{Motivation from a mathematical perspective}
By Remark \ref{correm}, solutions of gauge transforms correspond to solutions of autonomous systems. Autonomous systems have special properties, such as admitting a  local normalization (straightening near nonstationary points; Poincar\'e-Dulac normal forms near stationary points), and these -- mutatis mutandis -- carry over to gauge transforms. This correspondence will be elaborated in the following. Moreover the class of nonautonomous differential equations which are gauge transforms will be characterized  in some detail, and it will be shown that they are -- to some extent -- identifiable through explicit computations.

\section{Solution-preserving maps, symmetries}
We  turn to structural properties of gauge transforms.
\subsection{General properties}
We  discuss  the behavior with respect to solution-preserving maps.

Consider another autonomous analytic ODE
\begin{equation}\label{otherodeauto}
\dot x = g(x)
\end{equation}
on an open subset of $\mathbb R^n$, and its gauge transform $g^*(t,x)$ by $A(t)$.\\
\begin{proposition}\label{spmprop} Assume that a local diffeomorphism $\Phi$ sends solutions of \eqref{odeauto} to solutions of \eqref{otherodeauto},
and define
\[
\Gamma(t,x):=A(t)\Phi(A(t)^{-1}x).
\]
Then $\Gamma$ sends solutions of the gauge transform $f^*$ to solutions of the gauge  transform $g^*$:
If $z(t)$ solves \eqref{weirdgt}, then $\Gamma(t,z(t))$ solves $\dot y=g^*(t,y)$.
\end{proposition}
\begin{proof}  By a familiar criterion, the solution-preserving property holds if and only if the identity
\[
D\Phi(x)f(x)=g(\Phi(x))
\]
is satisfied.\\ For the proof first note 
\[
D_t\Gamma(t,x)=\dot A \Phi(A^{-1}x)-AD\Phi(A^{-1}x)A^{-1}\dot A A^{-1}x
\]
and 
\[
D_x\Gamma(t,x)=AD\Phi(A^{-1}x)A^{-1}.
\]
Then
\begin{eqnarray*}
\frac{d}{dt}\Gamma(t,z(t))&=&D_t\Gamma(t,z(t))+D_x\Gamma(t,z(t))\dot z(t)\\
    &=& \dot A\Phi(A^{-1}z){-AD\Phi(A^{-1}z)A^{-1}\dot A A^{-1}z}\\
    &+& AD\Phi(A^{-1}z)A^{-1}\left[ {\dot AA^{-1}z}+Af(A^{-1}z)\right]\\
 &=&\dot A\Phi(A^{-1}z)+AD\Phi(A^{-1}z)f(A^{-1}z)\\
&=&\dot A \Phi(A^{-1}z)+Ag(\Phi(A^{-1}z))\\
&=&\dot AA^{-1}\Gamma(t,z)+Ag(A^{-1}\Gamma(t,z))\\
&=&g^*(t,\Gamma(t,z))
\end{eqnarray*}
as asserted.
\end{proof}
Symbolically (in informal language) one may state Proposition \ref{spmprop} as follows:
\begin{equation}
\begin{matrix}
 & \Phi & \\
 f &\longrightarrow&g  \\
 & & \\
A\big\downarrow & & \big\downarrow A\\
& & \\
  f^* & \longrightarrow& g^* \\
 & \Gamma &
\end{matrix}
\end{equation}

\begin{corollary}\label{symcor}
If the local diffeomorphism $\Phi$ is a symmetry of \eqref{odeauto} (thus sends parameterized solutions to parameterized solutions), then $\Gamma$ is a symmetry of the gauge transform \eqref{weirdgt}.
\end{corollary}
Now consider the special setting when the diffeomorphism $\Phi$ is obtained from the local flow of a vector field $h$ with flow parameter $s$, thus
\begin{equation}\label{locflo}
\Phi(x)=\Phi_s(x)=H(s,x); \quad\text{with } \dfrac{\partial}{\partial s} H(s,x)=h(H(s,x)),\,H(0,x)=x.
\end{equation}

\begin{corollary}\label{symcor2}
Given $H$ as in \eqref{locflo}, with
\begin{equation}
\Gamma_s(t,x)= A(t)\Phi_s(t,A(t)^{-1}x)
\end{equation}
one has
\begin{equation}
\dfrac{\partial}{\partial s}\Gamma_s(t,x)|_{s=0}=A(t)h(A(t)^{-1}x)=:\widehat h(t,x);
\end{equation}
thus $\Gamma_s$ is the local flow of  $\widehat h(t,x)$, with $t$ to be considered a parameter.\\
In particular for $\Phi_s(t,x)=\exp(sB) x$ one finds $\Gamma_s$ as the local flow of \begin{equation}
\dfrac{\partial x}{\partial s}=A(t)BA(t)^{-1}x.
\end{equation}
\end{corollary}
\begin{remark}{\em 
Notably,  $\widehat h$ is not a gauge transform of $h$. 
Due to this fact, Lie bracket conditions for infinitesimal symmetries\footnote{{By an infinitesimal symmetry we mean a vector field whose local flow yields a local one-parameter group of symmetries.}}  do not transfer to Lie bracket conditions for gauge transforms:
Given vector fields $f$ and $h$, one has
\begin{equation}\label{liemix}
\left[\widehat h,\,f^*\right]_x= D_t\,\widehat h+ \widehat{[h,\,f]},
\end{equation}
where $[\cdot,\cdot]_x$ denotes the Lie bracket in $\mathbb R\times \mathbb R^n$ with respect to $x$, $t$ being a parameter.\\
In particular,
\[
[h,\,f]=0  \text{  iff  }\quad [\widehat h,\,f^*]_x=D_t \widehat h\quad\text{  iff  }\quad [\widehat f,\,h^*]_x=D_t \widehat f
\]
 The latter are familiar symmetry conditions, which say that
\[
\left[\begin{pmatrix}0\\ \widehat h\end{pmatrix},\,\begin{pmatrix}1\\ f^*\end{pmatrix}\right]=0\text{  resp.  } \left[\begin{pmatrix}0\\ \widehat f\end{pmatrix},\,\begin{pmatrix}1\\ h^*\end{pmatrix}\right]=0
\]
in $(t,x)$ coordinates.\\
To verify these assertions, combine 
\begin{equation*}
\begin{array}{rcl}
D_x\widehat h(t,x)&=&A\,Dh(A^{-1}x)\,A^{-1}\\
D_t\widehat h(t,x)&=&\dot A\, h(A^{-1}x)- A\,Dh(A^{-1}x)\,A^{-1}\dot A A^{-1}x
\end{array}
\end{equation*}
with
\begin{equation*}
\begin{array}{rcl}
D_xf^*(t,x)&=&\dot A A^{-1}+A\,Df(A^{-1}x)\,A^{-1}\\
\end{array}
\end{equation*}
and evaluate to get all the statements for $\widehat h$ and $f^*$. The remaining statements for the case $[h,\,f]=0$ follow by antisymmetry of the Lie bracket.}
\end{remark}

\subsection{Normalization}
{Local properties of the autonomous equation \eqref{odeauto} transfer to gauge transforms.}
\subsubsection{Straightening}
In the neighborhood of a nonstationary point $w$, with $f(w)=b\not=0$ there exists -- as well known -- a diffeomorphism $\Psi$ that maps solutions of \eqref{odeauto} to solutions of 
\[
\dot x=b
\]
with constant right hand side.
Thus, $\Delta(t,x) =A(t)\Psi(A(t)^{-1}x)$ sends solutions of the gauge transform $\dot x=f^*(t,x)$ to solutions of the gauge transform
\begin{equation}\label{nonautostr}
\dot y=\dot A\,A^{-1}y + Ab.
\end{equation}
Therefore, gauge transforms are locally equivalent to this class of nonautonomous linear equations. 
\subsubsection{Linear symmetries, Poincar\'e-Dulac }
When $\dot x=Bx$ is autonomous and linear, then $[B,f]=0$ if and only if $\Phi_s(x)=\exp(sB)x$ defines a symmetry of \eqref{odeauto} for every $s$. Consequently, by Corollary \ref{symcor2}, for every $s$,
\[
\Gamma_s(t,x):=A(t)\exp(sB)A(t)^{-1}x
\]
defines a symmetry of the gauge transform \eqref{weirdgt}.\\
Here it is worth recalling the special case when $f(0)=0$ and $B=Df(0)_{s}$, the semisimple part of the linearization, and a convergent transformation to normal form exists. (For more background see e.g. \cite{CW}.)

\subsection{Explicit solutions}
Finally we mention a rather obvious property of gauge transforms: If an explicit solution of \eqref{odeauto} is known, and the transformation $A(t)$ is given, then one obtains an explicit solution of the gauge transform. This observation may be interesting when a  nonautonomous equation can be identified as a gauge transform of an autonomous system. This is the topic of the next section.

\section{Identifying gauge transforms}
The close explicit correspondence between solutions of \eqref{odeauto} and of \eqref{weirdgt} makes a structural investigation of gauge transforms straightforward.\\
 However, in order to utilize special properties of gauge transforms, one needs tools to identify them among the class of nonautonomous differential equations. We turn to this problem now.  Thus for a given nonautonomous analytic system 
\begin{equation}\label{gennonauto}
\dot x =q(t,x)
\end{equation}
we consider the
{\bf identification problem:} 
\begin{center}
{\em Do there exist $A$ and $f$ so that \eqref{gennonauto} has the form \eqref{weirdgt}?} 
\end{center}
\subsection{Criteria}
We work with power series\footnote{We will consider convergent series, thus we require analyticity from here on. But most results are also valid in the formal case, or for Taylor approximations of finite order.} expansions
\begin{equation}\label{nonautops}
q(t,x)= c(t)+C(t)x+\sum_{j\geq 2}q_j(t,x), 
\end{equation}
with $c$ of degree zero, {$x\mapsto Cx$} homogeneous of degree one, and $q_j$ homogeneous of degree $j$ in $x$. Likewise,
\begin{equation}\label{autops}
f(x)=b+Bx+\sum_{j\geq 2}f_j(x)
\end{equation}
with $b$ constant, $B$ linear and $f_j$ homogeneous of degree $j$. \\

Before going into details, we observe that the linear part $C(t)$ can always  be removed from \eqref{nonautops}.
\begin{lemma}\label{linremove} Let  \eqref{gennonauto} be given, and let $T$ satisfy $\dot T=CT$, $T(0)=I$. If $u(t)$ is a solution of \eqref{gennonauto}, then $v(t):=T(t)u(t)$ solves the nonautonomous system
\begin{equation}\label{nongennonauto}
\dot y=\widetilde q(t,x)=c(t)+\sum_{j\geq2} T^{-1}q_j(t,\,Ty)=: c(t)+\sum_{j\geq 2}\widetilde q_j(t,y)
\end{equation}
with vanishing linear part.
\end{lemma}
\begin{proof}
Evaluate
\[
\dot Tu+T\dot u=\frac d{dt}(Tu)=c+CTu+\sum_{j\geq2}q_j(t,Tu).
\]
\end{proof}
The following proposition has an obvious proof: The solutions of all differential equations involved are related via multiplication by some (time dependent) invertible matrix. But one may  consider it quite relevant, since it allows (in theory) to set $C=0$.
\begin{proposition}
The nonautonomous vector field $q$ in \eqref{gennonauto} is a gauge transform of some autonomous system if and only if the nonautonomous vector field $\widetilde q$ in \eqref{nongennonauto} is.
\end{proposition}

We return to system \eqref{nonautops}.
Comparing homogeneous terms, we immediately have 
\begin{proposition}\label{degbydeg}
One has $q=f^*$ with some $A$ (given constant $B$) if and only if 
\begin{equation}\label{recconds}
\begin{array}{rcl}
c(t)&=&A(t)b\\
C(t)&=& \dot A(t)A(t)^{-1}+A(t)BA(t)^{-1}\\
q_j(t,x)&=& A(t)f_j(A(t)^{-1}x)
\end{array}
\end{equation}
for all $t,\,x$ and $j\geq 2$.
\end{proposition}
The degree one condition is crucial here. Rewriting and using standard facts about linear differential equations, we get:
\begin{lemma}\label{cruclin}
Given the linear part $C(t)$ of $q$, for every constant matrix $B$ there exists an invertible $A$ such that the degree one condition in Proposition \ref{degbydeg} is satisfied. A necessary and sufficient condition for $A$ is to satisfy the linear differential equation
\begin{equation}\label{Alindeq}
\dot A= CA-AB.
\end{equation}
\end{lemma}
\begin{remark}{\em 
There is a particular, and at first sight surprising,  consequence of this lemma: Every nonautonomous homogeneous linear differential equation $\dot x=C(t)x$ is a gauge transform of any autonomous homogeneous linear differential equation $\dot x=Bx$.}
\end{remark}

\begin{remark}\label{pararem}{\em We clarify the role of $A(0)=A_0$ and its relation to the ``parameter matrix'' $B$.
\begin{itemize}
\item The degree one condition shows that for given $C$, the possible transformations $A$ depend on $A_0$ and on $B$. From
\[
\dot A(0)=C(0)A_0-A_0B
\]
one sees that, conversely, 
\[
B=A_0^{-1}C(0)A_0-A_0^{-1}\dot A(0)
\]
 is determined by $A_0$ and $\dot A(0)$. To summarize, $A$ is determined by $C$, $A(0)$ and $\dot A(0)$.
\item One may normalize the invertible initial value $A_0$ to  the identity matrix $I$:  Let $A_*= AA_0^{-1}$. From the above relation for $A_0$ we get
\[
\dot A_*=\dot A A_0^{-1}=CAA_0^{-1}-ABA_0^{-1},
\]
so with $B_*:=A_0BA_0^{-1}$  we have
\[
\dot A_*=CA_*-A_*B_*,\quad A_*(0)=I
\]
with  remaining parameter $B_*$ (noting $\dot A_*(0)=C(0)-B_*$).
\end{itemize}
}
\end{remark}

We now use the criterion in Proposition \ref{degbydeg} and the subsequent observations to determine (as far as possible) an autonomous vector field $f$ such that $q$ is a gauge transform of $f$. Only the linear part $B$ remains as a parameter. For the proof of the following just compare terms of the same degree.
\begin{lemma}\label{wgtlem}
Consider the nonautonomous analytic system \eqref{gennonauto}, with
\eqref{nonautops} the power series expansion with respect to $x$. Then
\begin{equation}\label{wgtcond}
q(t,x)=\dot A A^{-1}x+ Af(A^{-1}x)
\end{equation}
with an autonomous system \eqref{odeauto}, and Taylor expansion \eqref{autops}, and invertible $A(t)$ with $A(0)=I$,
only if
\begin{itemize}
\item $\dot A(0)=C(0)-B$;
\item $ c(0)=b$;
\item $q_j(0,x)=f_j(x)$ for all $j\geq 2$.
\end{itemize}
\end{lemma}
One may restate Proposition \ref{degbydeg} without explicit reference to $f$:
\begin{proposition}\label{wgtprop}
The nonautonomous vector field $q$ as given in \eqref{gennonauto},  \eqref{nonautops} is a gauge transform  if and only if there exists $B$ such that
\begin{itemize}
\item $\dot A(t)=C(t)A(t)-A(t)B$; $A(0)=I$;
\item $ c(t)=A(t)c(0)$;
\item $q_j(t,x)=A(t)q_j(0,A(t)^{-1}x)$ for all $j\geq 2$.
\end{itemize}
\end{proposition}
\subsection{Vanishing linear part}
For systems  \eqref{gennonauto} with vanishing linear part, thus $C(t)$ identically zero,
we state a special consequence of Proposition \ref{wgtprop}.
\begin{corollary}\label{nolincor} Let $q$ be given with $C=0$. Then $q$ is a gauge transform of some $f$ via $A(t)$ if and only if the following identities hold:
\begin{itemize}
\item $A(t)=\exp (-tB)$ for some constant matrix $B$;
\item $c(t)=\exp(-tB)c(0)$;
\item $q_j(t,x)=\exp(-tB)q_j(0,\exp(tB)x)$
\end{itemize}
\end{corollary}
With these criteria we can easily (if belatedly) show that  gauge transforms are indeed distinguished among nonautonomous systems.
\begin{example}{\em  The system\footnote{This is only a quadrature problem, but here we want to exhibit the simplest example.}
\[
\dot x=c(t)
\]
 is the gauge transform of some $f$
by some $A(t)$ if and only if 
\[
c=\exp(-tB)\,b.
\]
Thus $c$ must be a a linear combination of products of polynomials and exponential functions. This is clearly not satisfied in general. }
\end{example}
We now discuss the third condition in Corollary \ref{nolincor} in a systematic manner. We set
\[
p_j(x):=q_j(0,x),\quad j\geq 2, 
\]
thus the condition may be restated as 
\[
q_j(t,x)=\exp(-tB)p_j(\exp(tB)x),\quad j\geq 2.
\]
Differentiation wirh respect to $t$ yields
\[
D_tq_j(t,x)=\exp(-tB)\left[B,\,p_j\right]\,\left(\exp(tB)x\right),
\]
and in particular
\begin{equation}\label{brackrest}
r_j(x):=D_tq_j(t,x)|_{t=0}=\left[B,\,p_j\right](x).
\end{equation}
We denote by ${\mathcal P}_{j-1}$  the (finite dimensional) space of homogeneous polynomial maps from $\mathbb R^n$ to $\mathbb R^n$,  of degree $j$. Thus  $p_j\in{\mathcal P}_{j-1}$, and $B\in\mathcal P_0$. (We may identify these spaces with their coefficient spaces  with respect to some basis of $\mathbb R^n$.) \\
The following proposition shows that the property of being a gauge transform is rather restrictive for nonlinear systems.
The proof uses results from R\"ohrl \cite{Rohrl}, and from \cite{KLPW}, Appendix. 
\begin{proposition}\label{restricpop} Consider  $q$ in \eqref{gennonauto} with linear part $C=0$.
Given $k\geq 2$, there exists a Zariski-open (and dense) subset $\mathcal Q_{k-1}$ of $\mathcal P_{k-1}$ such that for $p_k\in \mathcal Q_{k-1}$ equation \eqref{brackrest} has at most one solution. \\
Therefore, if $q$ is a gauge transform, then $B$ is uniquely determined by $q_k$ whenever $p_k\in \mathcal Q_{k-1}$.
\end{proposition}
\begin{proof}
We fix $m\geq 2$, and for some $p\in \mathcal P_{m-1}$, we consider the linear map
\begin{equation}\label{testmap}
\mathcal P_0\to\mathcal P_{m-1},\quad B\mapsto \left[B,\,p\right].
\end{equation}
In the following we may and will turn to complexifications.
The assertion of the proposition says that this map is injective for all $p$ in a Zariski-open and dense subset of $\mathcal P_{m-1}$. Let $B$ lie in the kernel of this map. Then for all $s$,  $\exp(sB)$ is an automorphism\footnote{By definition, an invertible linear map $T$ is an automorphism of $p$ if the identity $p(Tx)=T\,p(x)$ holds.} of $p$. \\
 Following \cite{Rohrl}, we call $c\in \mathbb C^n$ an {\em idempotent} of $p$ if
\[
p(c)=c\not=0.
\]
By \cite{Rohrl} and \cite{KLPW}, there exists a Zariski-open and dense subset $\mathcal Q_{m-1}$ of $\mathcal P_{m-1}$ such that every $p\in\mathcal Q_{m-1}$ has only finitely many idempotents, and moreover admits a basis of idempotents. Since automorphisms send idempotents to idempotents,  an automorphism of $p$ is uniquely determined by its action on idempotents. So, this group of automorphisms is finite, which implies $B=0$. Therefore \eqref{testmap} defines an injective map.
\end{proof}
\begin{remark}\label{uniqrem}{\em 
Given the identification of $\mathcal P_{m-1}$ with a coefficient space, equation\eqref{brackrest} amounts to a system of linear equations for the entries of $B$. Thus one is left with elementary linear algebra.\\
If $p\in\mathcal Q_{m-1}$  (in the notation of Proposition \ref{restricpop}), whenever a solution $B$ exists for \eqref{brackrest}, it is unique and it completely determines the nonautonomous system.}
\end{remark}
\begin{example} {\em To illustrate feasibility of the computations we consider 
\begin{equation}\label{quadex}
q(t,x)=q_2(t,x)=\begin{pmatrix} \left(1+ta_1(t)\right)x_1^2-\left(1+ta_2(t)\right)x_2^2\\  2\left(1+ta_3(t)\right)x_1x_2\end{pmatrix},
\end{equation}
which is homogeneous quadratic in $x$. Thus with the notation introduced above we have 
\[
p_2(x)=\begin{pmatrix} x_1^2-x_2^2\\  2x_1x_2\end{pmatrix},\quad  r_2(x)=\begin{pmatrix} a_1(0)x_1^2-a_2(0)x_2^2\\  2a_3(0)x_1x_2\end{pmatrix}.
\]
\begin{itemize}
\item We want to decide under which conditons on the $a_i(t)$, $q$ is a gauge transform. Thus we need to decide whether there exists a matrix $B=\begin{pmatrix} b_1&b_2\\ b_3&b_4\end{pmatrix}$ such that $\left[B,\,p_2\right]=r_2$, and determine this matrix if it exists. By straightforward calculations,
\[
\left[B,\,p_2\right](x)=\begin{pmatrix} b_1x_1^2-2b_3x_1x_2+(b_1-2b_4)x_2^2\\ b_3x_1^2+2b_1x_1x_2+(2b_2+b_3)x_2^2\end{pmatrix}.
\]
Therefore the space of all homogeneous quadratic vector fields
\[
\begin{pmatrix} c_1x_1^2+c_2x_1x_2+c_3x_2^2\\ c_4x_1^2+c_5x_1x_2+c_6x_2^2\end{pmatrix}\in \left[{\mathcal P}_1,\,p_2\right]
\]
is characterized by the conditions $2c_1-c_5=2c_2+c_2=0$, and  
\[
B=\begin{pmatrix} c_1 &\frac12(c_6-c_4)\\c_4& \frac12(c_1+c_3)\end{pmatrix}
\]
if these conditions are satified. (One can check that $p_2\in\mathcal Q_1$, but this is not necessary for the computations.)
For the special $r_2$ above we have the single condition $a_1(0)=a_3(0)$ for existence of $B$, and we find
\[
B=\begin{pmatrix}\lambda & 0\\ 0&\mu\end{pmatrix},\quad\text{with  }\lambda=a_1(0),\,\mu=a_1(0)+a_2(0).
\]
With Corollary \ref{nolincor} we arrive at
\[
q_2(t,x)=\begin{pmatrix} e^{-\lambda t}& 0 \\0 & e^{-\mu t}\end{pmatrix}\begin{pmatrix}(e^{\lambda t}x_1)^2 - (e^{\mu t}x_2)^2\\ 2e^{\lambda t}x_1e^{\mu t}x_2\end{pmatrix}=\begin{pmatrix} e^{\lambda t}x_1^2-e^{(2\mu-\lambda) t}x_2^2\\ 2e^{\lambda t}x_1x_2\end{pmatrix}.
\]
Thus all the gauge transforms among the systems \eqref{quadex} have been determined.
\item Since the solution of $\dot x=p_2(x)$ (which is the complex equation $\dot w=w^2$ written in real coordinates) can be determined explicitly, the same holds for the gauge transform found above. One finds, for initial value $v$, the solution
\[
\dfrac{1}{(1-tv_1)^2+(tv_2)^2}\begin{pmatrix}e^{-\lambda t}(v_1-t(v_1^2+v_2^2))\\ e^{-\mu t}v_2\end{pmatrix}.
\]
\item Looking beyond the homogeneous case, for any setting with
\[
q(t,x)=q_2(t,x)+\text{ higher order terms}
\]
we see from Remark \ref{uniqrem} that $B$ and $q$ are completely determined by $q_2$.
\end{itemize}
}
\end{example}
\section{Conclusion}
Generally, the theory of nonautonomous differential equations does not have a close connection to the more specialized theory of autonomous systems, even considering the familiar procedure to add time $t$ as a new dependent variable (which is of little help in qualitative investigations). From this vantage point it is of some interest to look into nonautonomous systems which have a closer relation to autonomous ones. \\
Special gauge transforms (via {time-dependent} linear transformations) of autonomous equations form such a distinguished class. We showed that the equations in this class possess special properties, for instance regarding symmetries and local normalization. \\
We characterized those nonautonomous equations which are gauge transforms of autonomous systems, and showed that identification is feasible by explicit computations for systems with vanishing linear part. For the general case one could say that Lemma \ref{linremove}, while obviously a valuable tool for theoretical considerations, also represents a stumbling block for explicit calculations because it is valid for completely arbitrary $C(t)$. But it is possible to proceed from this vantage point when $C(t)$ has special properties.

\subsection*{Acknowledgements}

{We thank the anonymous reviewers for their comments and suggestions.}\\The work of GG is partially supported by the project {\it ``Mathematical Methods in Non-Linear Physics''} (MMNLP) of INFN, and by GNFM-INdAM. In the course of this work GG enjoyed the hospitality and the relaxed work atmosphere of SMRI.


\end{document}